\documentclass[reqno,doublespacing]{amsart}

\usepackage[centertags]{amsmath}
\usepackage{amssymb}
\usepackage[english,francais]{babel}

\theoremstyle{plain} %default

\newtheorem{lem}{Lemma}

\newtheorem{thm}{Theorem}

\theoremstyle{definition}

\theoremstyle{remark}

\numberwithin{equation}{section}

\def\og{\leavevmode\raise.3ex\hbox{$\scriptscriptstyle\langle\!\langle$~}}
\def\fg{\leavevmode\raise.3ex\hbox{~$\!\scriptscriptstyle\,\rangle\!\rangle$}}

\DeclareMathOperator{\ad}{ad} \DeclareMathOperator{\Ad}{Ad}

\DeclareMathOperator{\Fract}{Fract}

\DeclareMathSymbol{\R}{\mathalpha}{AMSb}{"52}
\DeclareMathSymbol{\C}{\mathalpha}{AMSb}{"43}
\newcommand{\mbb}[1]{\mathbb{#1}}

\newcommand{\K}{\mbb{K}}
\newcommand{\beq}{\begin{equation}}
\newcommand{\eeq}{\end{equation}}

\newcommand{\set}[1]{\left\{#1\right\}}

\newcommand{\bd}{\begin{description}}
\newcommand{\ed}{\end{description}}
\newcommand{\beqr}{\begin{eqnarray}}
\newcommand{\eeqr}{\end{eqnarray}}
\newcommand{\beqt}{\begin{equation}}
\newcommand{\eeqt}{\end{equation}}

\begin{document}

\title{Casimir operators of
Lie algebras with a nilpotent radical}

\author{ Jean-Claude  Ndogmo}

\address{
PO Box 2446\\
Bellville 7535\\
 South Africa.}
\email{ndogmoj@yahoo.com}


\begin{abstract}
We show that a Lie algebra having a nilpotent radical has a
fundamental set of invariants consisting of Casimir operators. We
give a different proof of this fact in the special and well-known
case where the radical is abelian.
\end{abstract}

\subjclass[2000]{24, 56}
\maketitle

\section{Introduction}
\label{s:intro}
%%%%%%%%%%%%%%%%%%%%%%%%%%
%%%%%%%%%%%%%%%%%%%%%%%%

An important problem arising in the representation theory of Lie
groups is the determination of the invariants of a given
representation. If $\rho$ is a representation of the Lie group $G$
in a finite dimensional vector space $V,$ the invariants of $\rho$
are the elements $v$ of $V$ for which the equality $\rho(g)v=v$
holds for all $g \in G.$   A map $\rho$ is a representation of a
connected Lie group $G$ in the finite dimensional vector space $V$
 if and only if its differential $d \rho$ is a representation of the
Lie algebra $L$ of $G$ in $V.$ Moreover, for any $v\in V$ we have
$\rho(G)v=v$ if and only if $d \rho(L)v=0,$ and the latter condition
defines $v$ as an invariant of $d\rho.$ The invariants of $G$ and
$L$ are therefore the same and in general they are more easily
analyzed on Lie algebras. \par
   When $\rho$ is the adjoint representation $\Ad$ of $G$ in its Lie
algebra $L,$ the invariants of the corresponding representation of
$G$ in the symmetric algebra $S(L)$ are called the polynomial
invariants of $L.$ The algebra $S^I$ generated in $S(L)$ by the
invariant polynomial functions is algebraically isomorphic to the
algebra of Casimir operators \cite{dix59}, which is the center
$\mathcal{Z}(L)$ of the universal enveloping algebra
$\mathfrak{A}(L)$ of $L.$ If $\mathcal{B}= \set{v_1, \dots, v_n}$ is
a basis of the finite $n$-dimensional Lie algebra $L,$ with
structure constants $C_{ij}^k$ in this basis, it is well known
\cite{fom} that the infinitesimal generators $\tilde{v}_i$ of the
adjoint action of $G$ on $S(L)$ are operators of the form
$$ \tilde{v}_i= \sum_j \sum_k C_{ij}^k x_k \partial_j, $$
where $\set{x_1, \dots, x_n}$ is a coordinate system associated with
the basis $\mathcal{B},$ and an element $p \in S(L)$ is in $S^I$ if
and only if

\beq \label{eq:infivar} \tilde{v}_i\cdot p\;=\;0,  \qquad  \text{for
all $i=1, \dots, n$}. \eeq
This shows that the polynomial invariants of $L$ are completely
determined by its structure constants. A maximal set of functionally
independent solutions to \eqref{eq:infivar} for all possible types
of functions $p$ is referred to as a fundamental set of invariants
of $L.$ Thus isomorphic Lie algebras have the same fundamental sets
of invariants. We shall assume that the base field $\K$ of $L$ is of
characteristic zero.

\section{Characterization of the invariants}
\label{s:charater} We write the Levi decomposition of a given finite
dimensional Lie algebra $L$ in the form  $L= \mathcal{S}
\overset{.}{+} \mathcal{R},$ where $\mathcal{S}$ is the Levi factor
and the ideal $\mathcal{R}$ is the radical of $L.$
\begin{lem}
\label{l:isom} If the Lie algebra $L$ has a nilpotent radical, then
it is isomorphic to an algebraic Lie algebra.
\end{lem}

\begin{proof}
By Ado's Theorem, $L$ has a faithful representation $\phi$ in a
finite-dimensional vector space $V,$ in which elements in the
nilradical of $L$ are represented by nilpotent endomorphisms
\cite{jaco}. Since the radical $\mathcal{R}$ of $L$ is nilpotent, it
is equal to its nilradical, and the Lie algebra $\phi (\mathcal{R})$
which consists of nilpotent endomorphisms and is consequently
algebraic \cite{chev3}. Moreover, $\phi (\mathcal{R})$ is the
radical of $\phi (L),$ and a subalgebra of $\mathfrak{g l}(V)$ is
algebraic if and only if its radical is algebraic \cite{chev3}.
Consequently, $\phi (L)$ is algebraic.
\end{proof}

All semisimple and nilpotent Lie algebras belong to the class of Lie
algebras with a nilpotent radical. This is also the case for all
perfect Lie algebras, and more generally for derived subalgebras of
finite dimensional Lie algebras, precisely because the radical of
such Lie algebras is nilpotent \cite{jaco}.\par

By a result of \cite{abl}, every perfect Lie algebra, i.e. a Lie
algebras $L$ for which $[L, L]=L,$  has a fundamental set consisting
of polynomial invariants. However, we notice that this property also
holds for Lie algebras with an abelian radical. Indeed, write the
Levi decomposition of $L$ in the form
\begin{equation}\label{eq:pi}
L= \mathcal{S} \oplus_{\pi} \mathcal{R},
\end{equation}
 where $\pi$ is the restriction to
the semisimple Lie algebra $\mathcal{S}$ of the adjoint
representation of $L$ in the radical $\mathcal{R}.$ If
$\mathcal{R}^{\mathcal{S}}= \set{v\in \mathcal{R} \colon
\pi(\mathcal{S})v=0}$ is the set of invariants of this
representation, then because $\pi$ is semisimple, we have the direct
sum of vector space $\mathcal{R}= \mathcal{R}^\mathcal{S}
\overset{.}{+} [\mathcal{S}, \mathcal{R}].$ \vspace{2mm}
\begin{thm}
\label{th:abelian} Let $L= \mathcal{S} \oplus_{\pi} \mathcal{R}$ be
a Levi decomposition of $L.$
\begin{description}
\item[{\rm(a)}] The Lie algebra $L$  is perfect if and only if
$\mathcal{R}^\mathcal{S} \subseteq [\mathcal{R}, \mathcal{R}].$
\item[{\rm(b)}] If the radical $\mathcal{R}$ of $L$ is abelian, then $L$ has a
fundamental set of invariants consisting of polynomial functions.
\end{description}
\end{thm}

\begin{proof}
We know that $L$ is perfect if and only if $\mathcal{R}=
[\mathcal{S},\mathcal{R}] + [\mathcal{R},\mathcal{R}].$ Writing the
right hand side of this last equality as a direct sum
$[\mathcal{S},\mathcal{R}]\overset{.}{+} W \cap
[\mathcal{R},\mathcal{R}]$ of vector subspaces, where $W$ is a
complement subspace of $[\mathcal{S}, \mathcal{R}]$ in
$\mathcal{R},$ we see that $L$ is perfect if and only if
$\mathcal{R}^\mathcal{S} \subseteq [\mathcal{R}, \mathcal{R}],$
which proves part (a). For part (b) we note first that by a result
of \cite{ndog04}, if the representation $\pi$ does not possess a
copy of the trivial representation then $L$ is perfect, and the
result follows. If $\pi$ does have a copy of the trivial
representation, then $\mathcal{R}^\mathcal{S} \neq 0,$ and by part
(a) above, $L$ is not perfect. However, $L$ is in this case a direct
sum of the perfect ideal $\mathcal{S}\overset{.}{+} [\mathcal{S},
\mathcal{R}] $ and the abelian ideal $\mathcal{R}^\mathcal{S}.$ It
then follows again that all the invariants of $L$ can be chosen to
be polynomials.
\end{proof}

\begin{lem}
\label{l:rational} A Lie algebra with a nilpotent radical has a
fundamental set of invariants consisting of rational functions.
\end{lem}
\begin{proof}
Lemma \ref{l:isom} reduces the proof to the case of algebraic Lie
algebras, and the lemma readily follows from a result of J. Dixmier
\cite{dix2} asserting that any algebraic Lie algebra has a
fundamental set of invariants consisting of rational functions.\par
\end{proof}
   In the sequel we shall denote by $\Fract (A)$
the field of fractions of a Noetherian and integral ring $A.$ Set
$\mathfrak{K}(L)= \Fract (\mathfrak{A}(L)),$ and $\mathsf{K}(L)=~
\Fract(S(L)),$ and for each $x \in L,$ denote by
$\ad_{\mathsf{K}(L)} x$ the derivation of $\mathsf{K}(L)$ that
extends $\ad_L x$ and thus defines a representation of $L$ in
$\mathsf{K}(L).$ The invariants of this representation are called
the rational invariants of $L.$ It should be noted that
$\mathsf{K}(L)$ is isomorphic the field $\K (x_1, \dots, x_n)$ of
rational functions in $n$ commuting variables. Similarly, for each
$x \in L,$ denote by $\ad_{\mathfrak{K}(L)} x$ the derivation of
$\mathfrak{K}(L)$ that extends $\ad_L x.$ Finally, denote by
$\mathcal{C}(L)$ and $\mathcal{F}(L)$  the center of
$\mathfrak{K}(L)$ and $\mathsf{K}(L)$ respectively, when they are
endowed with the adjoint representation. The center $S^I$ of $S(L)$
is also given by $S^I= \set{p \in S(L)\colon \ad_{S(L)} (L) (p)=0},$
and we have $\Fract (S^I) \subset \mathcal{F}(L).$ By a result of
\cite{rent}, $\mathcal{C}(L)$ and $\mathcal{F}(L)$ are isomorphic
fields. Moreover, we have the following result \cite{abl} in which
$L^{\!*}$ denote the dual vector space of $L.$
\begin{lem}
\label{l:quotient} We have $f \in \mathcal{F}(L)$ if and only if
there exists some weight $\lambda \in L^{\!*}$ of the adjoint
representation of $L$ in $S(L)$ such that $f = p_1/p_2,$ where $p_1,
p_2 \in S(L)_{\lambda} = \set{p \in S(L)\colon \ad_{S(L)} x \,(p) =
\lambda(x)p, \text{ for all $x \in L$}} $
\end{lem}
\vspace{2mm}
   If $\lambda$ is in $L^{\!*},$ an element of
$S(L)_{\lambda}$ is called a $\lambda$-semi-invariant of $L$ in
$S(L).$  It is clear that if the weight space of any such $\lambda$
is not reduced to zero, then $\lambda$ defines a one dimensional
representation of $L$ in $\K,$ which vanishes on any perfect
subalgebra of $L$. \vspace{2mm}
\begin{thm}
\label{th:allcas} If the radical of the Lie algebra $L$ is
nilpotent, then $L$ has a fundamental set of invariants consisting
of Casimir operators.
\end{thm}
\begin{proof}
Since $L$ has a nilpotent radical, we may assume by Lemma
\ref{l:isom} that it is algebraic. It has therefore a fundamental
set of invariants that consists of rational invariants, by Lemma
~\ref{l:rational}.  In this case, we readily see that a sufficient
condition for $L$ to have only polynomial invariants is for the
equality $\mathcal{F}(L) = \Fract (S^I)$ to hold, and because of
Lemma ~\ref{l:quotient}, to prove this last equality it suffices to
verify that the only weight of the adjoint representation of $L$ in
$S(L)$ is 0. Consider the Levi decomposition $L= \mathcal{S}
\overset{.}{+}\mathcal{R},$ and let $\lambda$ be any weight of
$\ad_{S(L)}.$ If $x \in \mathcal{S},$ then clearly $\lambda (x)=0.$
On the other hand if $x \in \mathcal{R},$
 $ \ad_L x $ is nilpotent  and $\ad_{S(L)} x$ is locally nilpotent, and
hence $\lambda (x)=0.$ The rest of the theorem follows from the
isomorphism between $\mathcal{F}(L)$ and $\mathcal{C}(L).$
\end{proof}

Special cases of Theorem ~\ref{th:allcas} are known for $L$
semisimple \cite{dix}, $L$ nilpotent \cite{ber}, and L perfect
\cite{abl}. This theorem therefore unifies and extends seemingly
unrelated results asserting that the invariants for these particular
types of Lie algebras can all be chosen as Casimir operators, and
shows that the only Lie algebras that may not have a fundamental set
of invariants consisting of Casimir operators are to be found only
among the non nilpotent solvable Lie algebras.

\end{document}